\documentclass[11pt]{amsart}

\usepackage{amsfonts,amssymb,amsmath,amsthm,latexsym,graphics,epsfig}
\usepackage{verbatim,enumerate,array,booktabs,color,bigstrut,prettyref,tikz-cd}
\usepackage{multirow}
\usepackage[all]{xy}
\usepackage{hyperref}
\usepackage{ctable}

\usepackage{mathtools}

\newrefformat{eq}{\textup{(\ref{#1})}}
\newrefformat{prty}{\textup{(\ref{#1})}}

\newtheorem{defn}{Definition}

\newtheorem{lemma}[defn]{Lemma}

\newtheorem{thm}[defn]{Theorem}

\newtheorem{proposition}[defn]{Proposition}

\theoremstyle{definition}

\newcommand{\Q}{\mathbb Q}
\newcommand{\Qbar}{\overline{\Q}}

\newcommand{\modQ}{\,\text{mod}\,(\Q^*)^2}
\newcommand{\modQh}{\,\text{mod}\,(\Q[h])^2}

\newcommand{\eclabel}[1]{\href{https://www.lmfdb.org/EllipticCurve/Q/#1}{\texttt{#1}}}

\author{Enrique Gonz\'alez-Jim\'enez}
\address{ Departamento de Matem\'aticas, Universidad Aut\'onoma de Madrid,Madrid, Spain\,
\\
}
\email{enrique.gonzalez.jimenez@uam.es}
\urladdr{\url{http://www.uam.es/enrique.gonzalez.jimenez}}

\title[Isogeny vs. discriminant]{On the squareness of the discriminant\\  of elliptic curves with an isogeny}

\thanks{The author is supported by Grant PID2022-138916NB-I00 funded by MCIN/AEI/ 10.13039/501100011033 and by ERDF A way of making Europe.}

\begin{document}

\date{\today}
	
	\begin{abstract}
	We establish a classification of the values of \( N \) for which an elliptic curve defined over \( \mathbb{Q} \) with square discriminant admits an \( N \)-isogeny. Furthermore, we determine the values of \( N \) for which two elliptic curves defined over \( \mathbb{Q} \), both possessing square discriminants, are \( N \)-isogenous. In both cases, we explicitly parametrize the corresponding \( j \)-invariants of the elliptic curves associated with these problems.  
	\end{abstract}
	
	\maketitle

	\section{Introduction}
{
Let $E$ be an elliptic curve defined over $\Q$, given by a Weierstrass equation  
\[
E \colon y^2 + a_{1}xy + a_{3}y = x^3 + a_{2}x^2 + a_{4}x + a_{6}.
\]
The discriminant of $E$, denoted $\Delta(E)$, is defined in terms of the coefficients $a_1,a_2,a_3,a_4,a_6$ (see \cite[Chapter~III.1]{silverman}). Any other Weierstrass equation for $E$ can be obtained via the change of variables  
\[
\begin{array}{l@{\,=\,}l}
x & u^2x' + r, \\
y & u^3y' + u^2sx' + t,
\end{array}
\]
with $u,r,s,t \in \Q$ and $u \neq 0$. If $E'$ denotes the curve given by such a transformed equation, then their discriminants satisfy  
\[
\Delta(E) = u^{12}\Delta(E').
\]
Consequently, the property of $\Delta(E)$ being a square in $\Q$ is intrinsic to the elliptic curve $E$ and does not depend on the chosen Weierstrass equation. In other words, the statement $\sqrt{\Delta(E)} \in \Q$ is a property of the curve itself. }

Our goal in this paper is to establish a systematic classification of the conditions under which the discriminant of an elliptic curve within a specific family is a perfect square. Understanding this property is crucial, as it has significant implications for the arithmetic of elliptic curves.

 In \cite{irene}, we addressed this problem by analyzing the structure of the torsion subgroup. In particular, one of the main results in the aforementioned work established that if \( E \) is an elliptic curve defined over \( \mathbb{Q} \) and \( \sqrt{\Delta(E)} \in \mathbb{Q} \), then \( |E(\mathbb{Q})_{\mathrm{tors}}| \in \{1,3\} \) or \( E(\mathbb{Q})_{\mathrm{tors}} \) is non-cyclic. However, the converse holds only when \( E(\mathbb{Q})_{\mathrm{tors}} \) is non-cyclic. In this paper, we focus our attention on elliptic curves defined over $\Q$ that admit (cyclic) isogenies defined over \( \mathbb{Q} \). { Throughout this paper, we adopt the convention that isogenous elliptic curves are regarded as non-isomorphic.}
 
\

Our main results in this paper are the following three theorems.

\begin{thm}\label{mainthm1}	
Let $E$ be an elliptic curve defined over $\Q$ such that $\sqrt{\Delta(E)}\in \Q$. If $E$ admits an $N$-isogeny, then $N\in \{2,3,4,6,7,8\}$. The following table provides a parametrization of the \( j \)-invariant of \( E \):
\begin{center}
\begin{tabular}{cll}
\toprule %
$N$ &   $j(E)$ \\
\toprule %
$2$ & $\frac{256 \left(t^2+t+1\right)^3}{t^2 (t+1)^2}$ \\
\midrule
$3$ & $\frac{\left(t^2+3\right)^3 \left(t^2+27\right)}{t^2}$   \\
\midrule
$4$ & $\frac{256 \left(t^4+4 t^3+5 t^2+2 t+1\right)^3}{t^2 (t+1)^4 (t+2)^2}$  \\
\midrule
$6$ & $\frac{256 \left(t^2+3 t+3\right)^3 \left(t^6+9 t^5+30 t^4+45 t^3+30 t^2+9 t+3\right)^3}{t^2 (t+1)^6 (t+2)^6 (t+3)^2 (2 t+3)^2}$   \\
\midrule
$7$ & $\frac{\left(t^2-t+7\right) \left(t^2+t+7\right) \left(t^4+5 t^2+1\right)^3}{t^2}$  \\
\midrule
$8$ & $\frac{256 \left(t^8+8 t^7+28 t^6+56 t^5+69 t^4+52 t^3+22 t^2+4 t+1\right)^3}{t^2 (t+1)^8 (t+2)^2 \left(t^2+2 t+2\right)^2}$   \\
\bottomrule
\end{tabular}
\end{center}
\end{thm}
\begin{thm}\label{mainthm2}	
Let $E$ and $E'$ be two elliptic curves defined over $\Q$ such that $\sqrt{\Delta(E)},\sqrt{\Delta(E')}\in \Q$. If $E$ and $E'$ are $N$-isogenous, then $N\in \{2,3,4,7\}$. The following table provides a parametrization of the \( j \)-invariants of \( E \) and  $E'$:
\begin{center}
\begin{tabular}{cllll}
\toprule %
$N$ &    $j(E)$ & & &  $j(E')$\\
\toprule
$2$ &  $\frac{256 \left(t^4-t^2+1\right)^3}{(t-1)^2 t^4 (t+1)^2}$ & & & $\frac{16 \left(t^4+14 t^2+1\right)^3}{(t-1)^4 t^2 (t+1)^4}$  \\
\midrule
$3$ & $\frac{\left(t^2+3\right)^3 \left(t^2+27\right)}{t^2}$ & & & $\frac{\left(t^2+27\right) \left(t^2+243\right)^3}{t^6}$  \\
\midrule
$4$ & $\frac{256 \left(t^8-t^4+1\right)^3}{(t-1)^2 t^8 (t+1)^2 \left(t^2+1\right)^2}$ & & & $\frac{4 \left(t^8+60 t^6+134 t^4+60 t^2+1\right)^3}{(t-1)^8 t^2 (t+1)^8 \left(t^2+1\right)^2}$  \\
\midrule
$7$ & $\frac{\left(t^2-t+7\right) \left(t^2+t+7\right) \left(t^4+5 t^2+1\right)^3}{t^2}$ & & & $\frac{\left(t^2-t+7\right) \left(t^2+t+7\right) \left(t^4+245 t^2+2401\right)^3}{t^{14}}$  \\
\bottomrule
\end{tabular}
\end{center}
\end{thm}

\noindent {\bf Remark.} 
In the proofs of Theorems \ref{mainthm1} and \ref{mainthm2}, we observe that if \( E \) and \( E' \) are two elliptic curves defined over \( \mathbb{Q} \) that are \( 3 \)-isogenous or \( 7 \)-isogenous, then  
\[
\sqrt{\Delta(E)} \in \mathbb{Q} \quad \text{if and only if} \quad \sqrt{\Delta(E')} \in \mathbb{Q}.
\]

\

The last main result of this article focuses on the case of complex multiplication. This theorem provides a characterization of elliptic curves over \( \mathbb{Q} \) with complex multiplication whose discriminant is a square. In particular, this result asserts that there do not exist two elliptic curves defined over \( \mathbb{Q} \) that are isogenous and have complex multiplication with square discriminants.

\begin{thm}\label{thm_CM}	
Let $E$ be an elliptic curve defined over $\Q$ with complex multiplication. Then $\sqrt{\Delta(E)}\in \Q$ if and only if $E$ is $\Q$-isomorphic to $E_{-t^2}:y^2=x^3-t^2x$ for some $t\in \Q^*$. In particular, if $E'$ is an elliptic curve defined over $\Q$ isogenous to $E$ then $\sqrt{\Delta(E')}\notin \Q$.
\end{thm}

\

\subsection*{Acknowledgements}
This project was suggested by Arman Zargar, so we would like to thank him for the initial questions that inspired this work. { We would also like to thank the referees for their many detailed comments
that have helped improve the paper.}

\

\section{Twists, isogenies and complex multiplication.}
In this section, we introduce some fundamental concepts from the theory of elliptic curves that will allow us to establish arithmetic criteria for determining whether the discriminant of an elliptic curve defined over \( \mathbb{Q} \) is a square (see Proposition~\ref{prop_delta}). At the end of this section, we will provide a proof of Theorem ~\ref{thm_CM}.

\subsection{Twists}\label{sec:twists} 
Let $E:y^2=x^3+Ax+B$ be an elliptic curve defined over $\Q$. We denote by $\Delta(E)$ and by $j(E)$ its discriminant and $j$-invariant, respectively:
$$
\Delta(E)=-16 \left(4 A^3+27 B^2\right)\quad\mbox{and}\quad j(E)=-1728\frac{(4A)^3}{\Delta(E)}.
$$
The $j$-invariant of $E$ is an invariant of the $\Qbar$-isomorphism class of the curve, and it does not depend on the particular equation chosen (cf. \cite[III Prop.1.4]{silverman}). Moreover, any elliptic curve defined over $\Q$ isomorphic over $\Qbar$ to $E$ has a short Weierstrass model of the form:
$$
\begin{array}{lclcl}
\mbox{(i)} & & E^d:y^2=x^3+d^2Ax+d^3B & & \mbox{if $j(E)\ne 0,1728$,} \\
\mbox{(ii)} & & E^d:y^2=x^3+d Ax & & \mbox{if $j(E) =1728$,}\\
\mbox{(iii)} & & E^d:y^2=x^3+d B & & \mbox{if $j(E) = 0$,}
\end{array}
$$ 
where $d$ is an integer in $\Q^*/(\Q^*)^{n(E)}$ and $n(E)=2$ (resp. $4$ or $6$) if $j(E)\ne 0,1728$ (resp. $j(E)= 1728$ or $j(E)= 0$) (cf. \cite[X \S 5]{silverman}). The elliptic curve $E^d$ is called the $d$-twist of $E$, and in the particular case $j(E)\ne 0,1728$ it is called the $d$-quadratic twist of $E$. Note that we have

$$
\Delta(E^d)=\left\{
\begin{array}{ccl}
d^6 \Delta(E) & & \mbox{if $j(E)\ne 0,1728$,} \\
d^3 \Delta(E) & & \mbox{if $j(E) =1728$,}\\
d^2 \Delta(E) &  & \mbox{if $j(E) = 0$.}
\end{array}
\right.
$$

{
The following result plays a key role in the proofs of Theorems~\ref{mainthm1}, \ref{mainthm2}, and~\ref{thm_CM}. 
Although it is elementary and possibly well-known, I have not been able to locate 
a precise reference in the literature; hence, its proof is included here.
}
\begin{proposition}\label{prop_delta}
Let $E$ be an elliptic curve defined over $\Q$. Then $\sqrt{\Delta(E)}\in \Q$ if and only 
\begin{itemize}
\item $j(E)=t^2+1728$, for some $t\in\Q^*$, or
\item $j(E)=1728$ and $E$ is $\Q$-isomorphic to $y^2=x^3-s^2 x$, for some $s\in \Q^*$. 
 \end{itemize}
\end{proposition}
\begin{proof}
First assume $j_0=j(E)$, with $j_0\ne 0,1728$, and consider the elliptic curve
$$
E_{j_0}\,:\,y^2=x^3-27 \frac{j_0}{j_0-1728}x+54\frac{j_0}{j_0-1728}.
$$
A straightforward calculation yields 
$$
j(E_{j_0})=j_0\qquad \mbox{and}\qquad \Delta(E_{j_0})=2^{12}\cdot 3^{12}\frac{j_0^2}{(j_0-1728)^3}.
$$
Therefore there exists $d\in \Q^*/(\Q^*)^{2}$ such that $E$ is $\Q$-isomorphic to $E_{j_0}^d$. In particular, 
$$
\Delta(E)=\Delta(E_{j_0}^d)\equiv (j_0-1728)\modQ.
$$
Thus, $\sqrt{\Delta(E)}\in\Q$ if and only if $j(E)=1728+t^2$ for some $t\in \Q$. In particular $t\ne 0$ since $j(E)\ne 1728$.

Now assume $j(E)=0$ and consider the elliptic curve $E_0\,:\,y^2=x^3+1$ that satisfies $j(E_0)=0$ and $\Delta(E_0)=-3\cdot 12^2$. That is, $E$ is $\Q$-isomorphic to $E_{0}^d$ for some $d\in \Q^*/(\Q^*)^{6}$. In particular
$$
\Delta(E)=\Delta(E_{0}^d)\equiv -3\modQ.
$$
This proves $\sqrt{\Delta(E)}\notin\Q$ if $j(E)=0$.

Finally assume $j(E)=1728$ and consider the elliptic curve $E_{1728}\,:\,y^2=x^3+x$ that satisfies $j(E_{1728})=1728$ and $\Delta(E_{1728})=-8^2$. That is, $E$ is $\Q$-isomorphic to $E_{1728}^d$ for some $d\in \Q^*/(\Q^*)^{4}$. In particular
$$
\Delta(E)=\Delta(E_{1728}^d)\equiv -d \modQ.
$$
Therefore, $\sqrt{\Delta(E)}\in\Q$ if and only if $-d=s^2$ for some $s\in \Q^*$.
\end{proof}

\subsection{Isogenies over $\Q$.}
Let $E$ be an elliptic curve defined over $\Q$ and let $N\in\mathbb Z$, $N>1$. We say that $E$ admits a (rational) $N$-isogeny if it admits a (cyclic) isogeny $\phi$, of degree $N$, such that ${}^\sigma \phi=\phi$ for any $\sigma \in \mathrm{Gal}(\Qbar/\Q)$. Equivalently, $E$ has a rational (cyclic) subgroup of order $N$.

The complete classification of rational isogenies was completed due to work by Fricke, Kenku, Klein, Kubert, Ligozat, Mazur  and Ogg, among others (see the summary in \cite[Table 1 and 2]{lozano}):
	
\begin{thm}\label{thm-isog} Let $E$ and $E'$ be two $N$-isogenous elliptic curves defined over $\Q$. Then
\begin{itemize}
\item[(i)] $N\in\{2,3,4,5,6,7,8,9,10,12,13, 16,18,25\}$. In this case, there are infinitely many isomorphism classes and Table \ref{table1} (in the Appendix) shows the $j$-invariants of $E$ and $E'$, or
\item[(ii)] $N\in\{11,14,15,17,19,21,27,37,43,67,163\}$. In this case, there are finitely many isomorphism classes and Table \ref{table2} (in the Appendix) shows the $j$-invariants of $E$ and $E'$.
\end{itemize} 
\end{thm}

Given a $\Q$-isogeny class of elliptic curves defined over $\Q$, its isogeny graph consists of a vertex for each elliptic curve in the isogeny class and an edge for each rational isogeny of prime degree between elliptic curves in the isogeny class, with the degree recorded as a label on the edge. The isogeny graphs of elliptic curves over $\Q$ first appeared in the so-called Antwerp tables \cite{antwerp}. Although the first proof (in press) seems to be due to Chiloyan and Lozano-Robledo \cite[Theorem 1.2]{CLR}. 

\subsection{Complex multiplication over $\Q$}
It is well known that there are $13$ isomorphism classes of elliptic curves defined over $\Q$ with complex multiplication (cf. \cite[A \S 3]{silverman2}). The following list gives these thirteen $j$-invariants:
$$
\mathcal{CM}=\left\{\begin{array}{c}
0,\quad\,
2^4\cdot 3^3\cdot 5^3,\quad\,
-2^{15}\cdot  3\cdot  5^3,\quad\,
2^6\cdot  3^3,\quad\,
2^3\cdot  3^3\cdot  11^3,\quad\,
-3^3\cdot  5^3,\quad\,
 3^3\cdot  5^3\cdot  17^3,\quad\,\\
 \!\! 2^6\cdot  5^3,\quad\,
 -2^{15},\quad\,
-2^{15}\cdot  3^3,\,\,
 -2^{18}\cdot  3^3\cdot  5^3,\,\,\,
 -2^{15}\cdot  3^3\cdot  5^3\cdot  11^3,\,\,\,\,\,
 -2^{18}\cdot  3^3\cdot  5^3\cdot  23^3\cdot  29^3
 \end{array} \!\! \right\}.
 $$

\subsubsection*{Proof of Theorem \ref{thm_CM}.}
Assume that $j(E)\in \mathcal{CM}\setminus\{1728\}$. 
Then a straightforward computation shows that $j_0-1728$ is not a square. By Proposition \ref{prop_delta} we have $\sqrt{\Delta(E)}\notin \Q$. Now if $j(E)=1728$, Proposition \ref{prop_delta} shows that $\sqrt{\Delta(E)}\in \Q$ if and only if $E$ is $\Q$-isomorphic to $E_{-t^2}:y^2=x^3-t^2x$ for some $t\in \Q^*$. Finally the isogeny graph associated to the $\Q$-isogeny class of $E_{-t^2}$ is (cf. \cite[Table 5]{CLR}):  
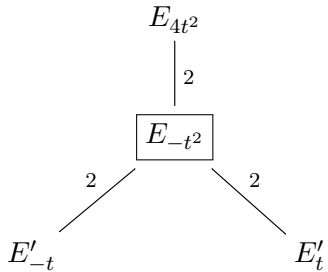
$$
\xymatrix{
& E_{4t^2} \ar@{-}[d]^2&  \\
& \fbox{$ E_{-t^2}$} \ar@{-}[dr]^2 \ar@{-}[dl]_2& \\
 E'_{-t} &  & E'_t }
$$
where 
 $E_{4t^2}:y^2=x^3+4t^2x$ and $E'_{\pm t}:y^2=x^3-11t^2x\pm 14t^3$. Note that $j(E_{4t^2})=1728$ and $j(E_{\pm t'})=2^3\cdot  3^3\cdot  11^3\in\mathcal{CM}$. We conclude the proof since $\Delta(E_{4t^2}) = -2^{12} t^6$ and $\Delta(E'_{\pm t}) = 2^9 t^6$ are not perfect squares.

\section{Proof of Theorem \ref{mainthm1} and \ref{mainthm2}}

{ The property of having complex multiplication is preserved under isogenies. In other words, if \(E\) and \(E'\) are isogenous elliptic curves, then either both have CM or neither does.} The case of complex multiplication has been addressed in Theorem \ref{thm_CM}. Therefore, from this point onward, we assume that all elliptic curves under consideration do not have complex multiplication. For the remainder of the proof, let \( E \) and \( E' \) be two \( N \)-isogenous elliptic curves defined over \( \mathbb{Q} \), both without complex multiplication. In particular, \( N \) does not belong to \(\{ 14, 19, 27,43, 67, 163\}\) (cf. Table \ref{table2} in the Appendix).

{
\begin{lemma}
Let $E$ be an elliptic curve defined over $\Q$. If $E$ admits an $N$-isogeny, with $N\in  \{11, 15,$ $ 17, 21,  37\}$, then \( \sqrt{\Delta(E)} \notin \mathbb{Q} \).
\end{lemma}
\begin{proof}
Let \( J_N \) denote the set of \( j \)-invariants of elliptic curves over \( \mathbb{Q} \) that admit an \( N \)-isogeny. By Theorem \ref{thm-isog}, these sets are finite, and they are listed in Table \ref{table2} (in the Appendix). Let \( j_0=j(E) \in J_N \). A straightforward computation then shows that \( j_0 - 1728$ is not a square. From Proposition \ref{prop_delta}, it follows that \( \sqrt{\Delta(E)} \notin \mathbb{Q} \).
\end{proof}
}

{
Let \(N \in \{2,3,4,5,6,7,8,9,10,12,13,16,18,25\}\),  and let \(E, E'\) be \(N\)-isogenous elliptic curves defined over $\Q$. Then there exists \(h \in \mathbb{Q}\) such that $j(E)=j_N(h)$ and $j(E')=j_N'(h)$ (see Table~\ref{table1} in the Appendix). 

Let \( F_N(h), G_N(h) \in \mathbb{Q}[h] \) be the polynomials defined in Table \ref{conditions}. Then we have 
$$
 F_N(h) \equiv (j_N(h) - 1728) \modQh \qquad\text{and}\qquad  G_N(h) \equiv (j'_N(h) - 1728) \modQh.
 $$}
 {
We define the curves
\[
C_N : y^2 = F_N(h) \quad \text{and} \quad X_N : \left\{ (h,y,z) \,:\, y^2 = F_N(h), z^2 = G_N(h) \right\},
\]
}
and the sets
\[
\begin{array}{l}
C^*_N(\mathbb{Q}) := \{ h \, : \, (h, y) \in C_N(\mathbb{Q}) \text{ and } j_N(h) \neq \infty \}, \\[1mm]
X^*_N(\mathbb{Q}) := \{ h \, : \, (h, y, z) \in X_N(\mathbb{Q}) \text{ and } j_N(h), j'_N(h) \neq \infty \}.
\end{array}
\]
{
By Proposition \ref{prop_delta}, the curve \( C_N \) parametrizes elliptic curves that admit an \( N \)-isogeny and have square discriminants, whereas the curve \( X_N \) parametrizes pairs of \( N \)-isogenous elliptic curves with square discriminants. In particular:
\begin{itemize}
    \item There exists an elliptic curve \(E\) defined over \(\Q\) that admits an \(N\)-isogeny with \(\sqrt{\Delta(E)} \in \Q\) if and only if \(C^*_N(\Q)\neq\emptyset\).
    \item There exist two \(N\)-isogenous elliptic curves \(E\) and \(E'\) defined over \(\Q\) such that $\sqrt{\Delta(E)},$ $\sqrt{\Delta(E')} \in \Q\) if and only if \(X^*_N(\Q)\neq\emptyset\).
\end{itemize}
}

\begin{table}
\begin{tabular}{llll}
\toprule
$N$ & $F_N(h)$  && $G_N(h)$ \\
\toprule 
$2$ & $h(h+64)$ && $h+64$ \\
\midrule
$3$ & $h$ &&  $h$ \\
\midrule
$4$ & $h(h+16)$ && $h+16$ \\
\midrule
$5$ & $h(h^2+22h+125)$ && $h(h^2+22h+125)$ \\
\midrule
$6$ & $h(h+8)$ && $h+9$ \\
\midrule
$7$ & $h$ &&  $h$ \\
\midrule
$8$ & $h^2-16$ && $h$ \\
\midrule
$9$ & $h^3-27$ && $h^3-27$ \\
\midrule
$10$ & $h(h-4)(h^2+4)$ && $(h+1)(h^2+4)$ \\
\midrule
$12$ & $(h^2-1)(h^2-9)$ && $h(h-1)(h+3)$ \\
\midrule
$13$ & $h(h^2+6h+13)$ && $h(h^2+6h+13)$ \\
\midrule
$16$ & $(h^2-4)(h^2+4)$ && $h(h^2+4)$ \\
\midrule
$18$ & $h(h^3-8)$ && $h^3+1$ \\
\midrule
$25$ & $(h-1) \left(h^2+4\right) \left(h^4+h^3+6 h^2+6 h+11\right)$ && $(h-1) \left(h^2+4\right) \left(h^4+h^3+6 h^2+6 h+11\right)$\\
\bottomrule
\end{tabular}
\caption{}\label{conditions}
\end{table}

{
\noindent {\bf Remark.} 
Based on results of Rouse–Zureick-Brown \cite{RZB} and Zywina \cite{zywina}, 
Morrow \cite{morrow} constructed models of composite level modular curves whose rational points 
parametrize elliptic curves over $\Q$ with simultaneously non-surjective composite Galois images. Some of the curves obtained in his work turn out to be equivalent to some of those arising in the present article, 
although they appear here in a completely different context.
}

\subsection*{Proof of Theorem \ref{mainthm1}.}

First, we observe that our investigation is restricted to the set of values \( N \in \{2, 3, 4, 5, 6, 7, 8, 9, 10, 12, 13, 16, 18, 25\} \). Table \ref{CN} provides essential information regarding the curve \( C_N \) for each specified \( N \).

\begin{table}[ht]
\begin{tabular}{ccclc}
\toprule
$N$ & $g(C_N)$  & $h$  & Rational points & LMFDB\\
\toprule %
$2$ & $0$ & $16\frac{t^2}{t+1}$  &    &\\
\midrule
$3$ &  $0$ &$t^2$ &     &\\
\midrule
$4$ &  $0$ &$4\frac{t^2}{t+1}$ &   &\\
\midrule
$5$ & $1$ & $-$ & $C_5(\Q)=\{(0,0)\}$ & \eclabel{20.a2}  \\
\midrule
$6$ & $0$ & $2\frac{t^2}{t+1}$ &  & \\
\midrule
$7$ &  $0$ &$t^2$ & & \\
\midrule
$8$ & $0$ & $2\frac{t^2+2t+2}{t+1}$ &  &\\
\midrule
$9$ & $1$ & $-$ & $C_9(\Q)=\{(3,0)\}$ & \eclabel{36.a3}  \\
\midrule
$10$ & $1$ & $-$ & $C_{10}(\Q)=\{(4, 0),  (0, 0), (-1,\pm 5)\}$ & \eclabel{20.a4}  \\
\midrule
$12$ & $1$ & $-$ & $C_{12}(\Q)=\{(0, \pm 3), (\pm 3, 0), (\pm 1, 0)\}\qquad$ &  \eclabel{24.a4}  \\
\midrule
$13$ & $1$ & $-$ & $C_{13}(\Q)=\{(0, 0)\}$ &   \eclabel{52.a2}  \\
\midrule
$16$ & $1$ & $-$ & $C_{16}(\Q)=\{(\pm 2, 0)\}$ &  \eclabel{32.a4}  \\
\midrule
$18$ & $1$ & $-$ & $C_{18}(\Q)=\{(0, 0), (2, 0),(-1, \pm 3)  \}$ & \eclabel{36.a4}  \\
\midrule
$25$ & $3$ & $-$ & $C_{25}(\Q)=\{(1,0) \}$ &   \\ 
\bottomrule
\end{tabular}
\caption{Data on the curve $C_N$}\label{CN}
\end{table}

{
Table \ref{CN} description: For each integer \( N \) in the set  \(\{2, 3, 4, 5, 6, 7, 8, 9, 10, 12, 13, 16, 18, 25\}\) in the first column,  the second column lists the genus of the curve \( C_N \). The third column provides the parametrization of \( h \) obtained from \( \phi_N \) in cases where the genus is \( 0 \); otherwise, it is marked as \lq\lq\( - \)\rq\rq. The fourth column displays the set \(C_N(\mathbb{Q})\) in the cases where the genus of \(C_N\) is greater than \(0\).  Finally, the last column provides the corresponding LMFDB label when \(C_N\) is an elliptic curve.
}

\

{\bf The case of $\text{genus}(C_N)=0$}. The curve $C_N\,:\,y^2=F_N(h)$ has genus $0$ for $N\in\{2,3,4,6,7,8\}$. Since $(0,0)\in C_N(\Q)$ for $N\in\{2,3,4,6,7\}$ and $(4,0)\in C_8(\Q)$, we can compute a parametrization $\phi_N\colon\mathbb A^1(\Q)\longrightarrow C_N(\Q)$.

For example for $N=2$, we have $C_2:y^2=h(h+64)$ and $\phi_2(t)=\left(\frac{16 t^2}{t+1},\frac{16 t (t+2)}{t+1}\right)$. Therefore
$$ 
C^*_2(\Q)=\left\{\frac{16 t^2}{t+1}\,:\, t\in\Q\right\}.
$$
 The cases for $N\in\{3,4,6,7,8\}$ are analogous, and the complete information appears in Table \ref{CN}.
 
\

{\bf The case of \(\text{genus}(C_N) = 1\)}. The curve \( C_N \colon y^2 = F_N(h) \) has genus 1 for $ N \in \{5, 9, 10, 12, 13, 16,$ $18\}$. Our objective is to compute the affine \(\mathbb{Q}\)-rational points on \( C_N \) to demonstrate that they correspond to the cusps of the map \( j_N(h) \).

Consider, for instance, the case \( N = 5 \). The curve parametrizing elliptic curves admitting a 5-isogeny with a square discriminant is given by:
\[
C_{5} \colon y^2 = h(h^2 + 22h + 125).
\]
This is an elliptic curve defined over \(\mathbb{Q}\), with the point \((0,0) \in C_5(\mathbb{Q})\). Its LMFDB label is \eclabel{20.a2} (see \cite{lmfdb}). The Mordell--Weil group of this elliptic curve is isomorphic to \(\mathbb{Z}/2\mathbb{Z}\), indicating that \( C_{5}(\mathbb{Q}) = \{(0,0)\} \).

Consequently, we have established that if \( E \) is an elliptic curve admitting a 5-isogeny over \(\mathbb{Q}\), then \( \sqrt{\Delta(E)} \notin \mathbb{Q} \).

The cases for \( N \in \{9, 10, 12, 13, 16, 18\} \) are analogous, and comprehensive information is provided in Table \ref{CN}.

\

{\bf The case of \(\text{genus}(C_N) > 1\)}. There exists only one curve \( C_N \) with genus greater than $1$, specifically \( C_{25} \). We can assert that \(C_{25}(\mathbb{Q}) = \{(1,0)\}\) and \(C^*_{25}(\mathbb{Q}) = \emptyset\), since we proved above that if 
\(E\) is an elliptic curve admitting a \(5\)-isogeny over \(\mathbb{Q}\), then \(\sqrt{\Delta(E)} \notin \mathbb{Q}\).  
This follows from the fact that if \(E\) admits a \(25\)-isogeny, then it necessarily admits a \(5\)-isogeny.

\subsection*{Proof of Theorem \ref{mainthm2}.}

In the first place, observe that, as a consequence of Theorem  \ref{mainthm1} , it suffices to study the cases where $N$ belongs to the set $N\in\{2,3,4,6,7,8\}$. Table \ref{XN} provides essential information regarding the curve \( X_N \) for each specified \( N \). Notably, the genus of \( X_N \) satisfies \( \text{genus}(X_N) \leq 1 \).

{\bf The case of $\text{genus}(X_N)=0$}. This case is analogous to the genus \( 0 \) case for the curve \( C_N \). The curve \( X_N \) has genus \( 0 \) for \( N \in \{2,3,4,7\} \). Since \( (0,0,0) \in X_N(\mathbb{Q}) \) for \( N \in \{3,7\} \), \( (0,0,8) \in X_2(\mathbb{Q}) \), and \( (0,0,4) \in X_4(\mathbb{Q}) \), we can compute a parametrization  
\[
\psi_N \colon\mathbb{A}^1(\mathbb{Q}) \longrightarrow X_N(\mathbb{Q}).
\]
For example, for \( N = 2 \), we have  
\[
X_2 : \{ (h,y,z)\,:\,y^2 = h(h+64), \quad z^2 = h+64 \},
\]  
and the parametrization  
\[
\psi_2(t) = \left( \left( 4 \frac{t^2-1}{t} \right)^2, 16 \frac{t^4-1}{t^2}, 4 \frac{t^2+1}{t} \right).
\]
Therefore,  
\[
X^*_2(\mathbb{Q}) = \left\{ \left( 4 \frac{t^2-1}{t} \right)^2 \,:\, t \in \mathbb{Q} \right\}.
\]
The cases for \( N \in \{3,4,7\} \) are analogous, and the complete details can be found in Table \ref{XN}.

{\bf The case of $\text{genus}(X_N)=1$}. This case is analogous to the genus \( 1 \) case for the curve \( C_N \). The curve \( X_N \) has genus \( 1 \) for \( N \in \{6,8\} \). Our goal is to compute the affine \( \mathbb{Q} \)-rational points on \( X_N \) to show that they correspond to cusps of the maps \( j_N(h) \) or \( j'_N(h) \).

Let us analyze the case \( N = 6 \) in detail. The curve \( X_6 \) is an elliptic curve defined over \( \mathbb{Q} \), as it is a genus $1$ curve containing the rational point \( (0,0,3) \in X_6(\mathbb{Q}) \). Its LMFDB label is \eclabel{24.a4}. The Mordell-Weil group of this elliptic curve is isomorphic to \( \mathbb{Z}/2\mathbb{Z} \oplus \mathbb{Z}/4\mathbb{Z} \). Therefore, we have  
\[
X_6(\mathbb{Q}) = \{ (-9, \pm 3, 0), (-8, 0, \pm 1), (0, 0, \pm 3) \}.
\]
Thus, we have proven that if \( E \) and \( E' \) form a pair of \( 6 \)-isogenous elliptic curves over \( \mathbb{Q} \), then \( \sqrt{\Delta(E)}, \sqrt{\Delta(E')} \notin \mathbb{Q} \).

The case \( N = 8 \) is analogous, and the complete details can be found in Table \ref{XN}.

\

\begin{table}[ht!]
\begin{tabular}{ccclc}
$N$ & $g(X_N)$  & $h$  & Rational Points & LMFDB\\
\toprule %
$2$ & $0$ & $\left(4\frac{t^2-1}{t}\right)^2$  &  &  \\
\midrule
$3$ &  $0$ &$t^2$ &   &  \\
\midrule
$4$ &  $0$ &$\left(2\frac{t^2-1}{t}\right)^2$ &  & \\
\midrule
$6$ & $1$ & $-$ &   $X_{6}(\Q)=\{ (-9, \pm 3, 0), (-8, 0, \pm 1), (0, 0, \pm 3)\}$ & \eclabel{24.a4}  \\
\midrule
$7$ &  $0$ &$t^2$ &  & \\
\midrule
$8$ & $1$ & $-$ &  $X_{8}(\Q)=\{ (4, 0, \pm 2)\}$ &\eclabel{32.a4}  \\
\bottomrule
\end{tabular}
\caption{Data on the curve $X_N$}\label{XN}
\end{table}

{ Table \ref{XN} description: This table is analogous to Table \ref{CN}. For each integer \( N \) in the set  \(\{2, 3, 4, 6,7, 8\}\) in the first column,  the second column lists the genus of the curve \( X_N \). The third column provides the parametrization of \( h \) obtained from \( \psi_N \) in cases where the genus is \( 0 \); otherwise, it is marked as \lq\lq\( -\)\rq\rq. The fourth column displays the set \(X_N(\mathbb{Q})\) in the cases where the genus of \(X_N\) is greater than \(0\).  Finally, the last column provides the corresponding LMFDB label when \(X_N\) is an elliptic curve.
}

\section*{Appendix: $j$-invariants such that $E$ and $E'$ are $N$-isogenous}
For the sake of completeness, we include in this appendix the \( j \)-invariants of pairs of isogenous elliptic curves over \( \mathbb{Q} \). These tables are primarily extracted from \cite{lozano}. Table \ref{table1} shows the \( j \)-map in cases where the modular curve \( X_0(N) \) has genus \( 0 \), while the remaining cases are listed in Table \ref{table2}.
\begin{table}[h]
\begin{tabular}{lll}
\toprule
$N$ & $j_N=j(E)$ & $j'_N=j(E')$ \\
\toprule
2 & $j_2(h)=\frac{(h+16)^3}{h}$ & $j'_2(h)=\frac{(h+256)^3}{h^2}$ \\
\midrule
3 & $j_3(h)=\frac{(h+27)(h+3)^3}{h}$ & $j'_3(h)=\frac{(h+27)(h+243)^3}{h^3}$ \\
\midrule
4 & $ j_4(h)=\frac{(h^2+16h+16)^3}{h(h+16)}$ & $j'_4(h)=\frac{(h^2+256h+4096)^3}{h^4(h+16)}$ \\
\midrule
5 & $ j_5(h)=\frac{(h^2+10h+5)^3}{h}$ & $j'_5(h)=\frac{(h^2+250h+5^5)^3}{h^5}$ \\
\midrule
6 & $j_6(h)=\frac{(h+6)^3(h^3+18h^2+84h+24)^3}{h(h+8)^3(h+9)^2}$ & $j'_{6}(h)=\frac{(h+12)^3(h^3+252h^2+3888h+15552)^3}{h^6(h+8)^2(h+9)^3}$   \\
\midrule
7 & $j_7(h)=\frac{(h^2+13h+49)(h^2+5h+1)^3}{h}$ & $j'_{7}(h)=\frac{(h^2+13h+49)(h^2+245h+2401)^3}{h^7}$ \\
\midrule
8 & $ j_8(h)=\frac{(h^4-16h^2+16)^3}{(h^2-16)h^2}$ & $j'_{8}(h)=\frac{(h^4 + 240h^3 + 2144h^2 + 3840h + 256)^3}{(h-4)^8h(h+4)^2}$ \\
\midrule
9 & $j_9(h)=\frac{h^3(h^3-24)^3}{h^3-27}$ & $j'_{9}(h)=\frac{(h+6)^3(h^3 + 234h^2 + 756h + 2160)^3}{(h-3)^8(h^3-27)}$ \\
\midrule
10 & $j_{10}(h)=\frac{(h^6-4h^5+16h+16)^3}{(h+1)^2(h-4)h^5}$ & $j'_{10}(h)=\frac{(h^6 + 236h^5 + 1440h^4 + 1920h^3 + 3840h^2 + 256h + 256)^3}{(h-4)^{10}h^2(h+1)^5}$\\
\midrule
\multirow{2}*{12} & \multicolumn{2}{l}{$j_{12}(h)=\frac{(h^2-3)^3(h^6-9h^4+3h^2-3)^3}{h^4(h^2-9)(h^2-1)^3}$} \\
\cmidrule(l){2-3} & \multicolumn{2}{l}{$j'_{12}(h)=\frac{(h^2 + 6h - 3)^3(h^6 + 234h^5 + 747h^4 + 540h^3 - 729h^2 - 486h - 243)^3}{(h-3)^{12}(h-1)h^3(h+1)^4(h+3)^3}$} \\
\midrule 
\multirow{2}*{13} & \multicolumn{2}{l}{$j_{13}(h)=\frac{(h^2+5h+13)(h^4+7h^3+20h^2+19h+1)^3}{h}$} \\
\cmidrule(l){2-3} & \multicolumn{2}{l}{$j'_{13}(h)=\frac{(h^2+5h+13)(h^4+247h^3+3380h^2+15379h+28561)^3}{h^{13}}$} \\
\midrule
\multirow{2}*{16} & \multicolumn{2}{l}{$j_{16}(h)=\frac{(h^8-16h^4+16)^3}{h^4(h^4-16)}$}\\ 
\cmidrule(l){2-3} & \multicolumn{2}{l}{$j'_{16}(h)=\frac{(h^8 + 240h^7 + 2160h^6 + 6720h^5 + 17504h^4 + 26880h^3 + 34560h^2+ 15360h + 256)^3}{(h-2)^{16}h(h+2)^4(h^2+4)}$} \\
\midrule
\multirow{2}*{18} & \multicolumn{2}{l}{$j_{18}(h)=\frac{(h^3-2)^3(h^9-6h^6-12h^3-8)^3}{h^9(h^3-8)(h^3+1)^2}$} \\
\cmidrule(l){2-3} & \multicolumn{2}{l}{$j'_{18}(h)=\frac{(h^3 + 6h^2 + 4)^3(h^9 + 234h^8 + 756h^7 + 2172h^6 + 1872h^5 + 3024h^4 + 48h^3 + 3744h^2 + 64)^3}{(h-2)^{18}h^2(h+1)^9(h^2-h+1)(h^2+2h+4)^2}$}  \\
\midrule
\multirow{2}*{25} & \multicolumn{2}{l}{$j_{25}(h)=\frac{(h^{10}+10h^8+35h^6-12h^5+50h^4-60h^3+25h^2-60h+16)^3}{h^5+5h^3+5h-11}$} \\ 
\cmidrule(l){2-3} & \multicolumn{2}{l}{$j'_{25}(h)=\frac{(h^{10} + 240h^9 + 2170h^8 + 8880h^7 + 34835h^6 + 83748h^5 +
206210h^4 + 313380h^3 + 503545h^2 + 424740h + 375376)^3}{(h-1)^{25}(h^4 + h^3 + 6h^2 + 6h + 11)}$} \\
\bottomrule
\end{tabular}
\caption{$j$-invariants such that $E$ and $E'$ are $N$-isogenous}\label{table1}
\end{table}

\begin{table}[h]
{
\begin{tabular}{llll}
\toprule
$N$ & $j_N=j(E)$ & $j'_N=j(E')$ & CM? \\
\toprule
 \multirow{2}{*}{$11$} & $-11\cdot 131^3$ & $-11^2$ & no \\
                                   & $-2^{15}$ & $-2^{15}$ & yes \\
\midrule
$14$  & $-3^3 \cdot 5^3$ & $3^3 \cdot 5^3 \cdot 17^3$ & yes \\
\midrule
 \multirow{2}{*}{$15$}  & $ -5^2/2$ & $  5 \cdot 211^3/2^{15}$ & no\\
  & $  -5^2 \cdot 241^3/2^3$ & $ -5 \cdot 29^3/2^5$ & no \\
\midrule
$17$  & $ -17^2 \cdot 101^3/2$ & $-17 \cdot 373^3/2^{17}$ & no \\
\midrule
$19$  & $-2^{15}\cdot 3^3$ & $-2^{15}\cdot 3^3$ & yes \\
\midrule
 \multirow{2}{*}{$21$}  & $-3^2 \cdot 5^6/2^3$ & $-3^3 \cdot 5^3\cdot 383^3/2^{7}$ & no\\
 & $3^3 \cdot 5^3/2$ & $ -3^2 \cdot 5^3\cdot 101^3/2^{21}$ & no\\
\midrule
$27$ & $ -2^{15} \cdot 3 \cdot 5^3$ & $ -2^{15} \cdot 3 \cdot 5^3$ & yes\\
\midrule
$37$  & $  -7 \cdot 11^3$ & $   -7 \cdot 137^3 \cdot 2083^3$  & no\\
\midrule
$43$ & $-2^{18} \cdot 3^3 \cdot 5^3$ & $-2^{18} \cdot 3^3 \cdot 5^3$ & yes \\
\midrule
$67$  & $ -2^{15}\cdot 3^3\cdot 5^3\cdot 11^3$ & $ -2^{15}\cdot 3^3\cdot 5^3\cdot 11^3$  & yes \\
\midrule
$163$  &$ -2^{18}\cdot 3^3\cdot 5^3\cdot 23^3\cdot 29^3\quad$ & $ -2^{18}\cdot 3^3\cdot 5^3\cdot 23^3\cdot 29^3\quad$ & yes  \\
\bottomrule
\end{tabular}
\caption{$j$-invariants such that $E$ and $E'$ are $N$-isogenous}\label{table2}
}
\end{table}

\end{document}